\documentclass[reqno]{amsart}
\usepackage{amssymb,amsthm,amscd}
\usepackage[pdftex]{graphics}
\input{diagrams}

\newcommand{\fa}{{\mathfrak a}}
\newcommand{\fA}{{\mathfrak A}}

\newcommand{\fp}{{\mathfrak p}}
\newcommand{\fP}{{\mathfrak P}}

\newcommand{\fm}{{\mathfrak m}}

\newcommand{\Cl}{{\operatorname{Cl}}}
\newcommand{\Gal}{{\operatorname{Gal}}}
\newcommand{\Spl}{{\operatorname{Spl}}}
\newcommand{\art}[2]{\big(\frac{#1}{#2}\big)}
\newcommand{\Z}{{\mathbb Z}}
\newcommand{\Q}{{\mathbb Q}}

\newcommand{\cO}{{\mathcal O}}

\newcommand{\too}{\longmapsto}
\newcommand{\lra}{\longrightarrow}
\newcommand{\la}{\langle}
\newcommand{\ra}{\rangle}

\newtheorem{thm}{Theorem}[section]
\newtheorem{prop}[thm]{Proposition}
\newtheorem{lem}[thm]{Lemma}
\newtheorem{cor}[thm]{Corollary}
\numberwithin{equation}{section}

\title{Harbingers of Artin's Reciprocity Law. \\
         III. Gauss's Lemma and Artin's Transfer}
\author{F. Lemmermeyer}
\email{hb3@ix.urz.uni-heidelberg.de}
\address{M\"orikeweg 1, 73489 Jagstzell, Germany}

\begin{document}
\maketitle

\markboth{Harbingers of Artin's Reciprocity Law}
         {\today \hfil Franz Lemmermeyer}
\begin{center} \today \end{center}
\bigskip

\section{Artin's Reciprocity Law}
This section is devoted to a brief presentation of Artin's 
reciprocity law in the classical ideal theoretic language.
The central part of this article is devoted to studying connections
between Artin's reciprocity law and the proofs of the quadratic
reciprocity law using Gauss's Lemma.

Let $\fm$ be a modulus (a formal product of an ideal $\fm_0$ and 
some infinite primes in some number field), and $D_\fm$ the
group of fractional ideals coprime to $\fm$. An ideal group $H$
is a satisfying $P_\fm^{(1)} \subseteq H \subseteq D_\fm$. To 
each finite extension $K/F$ and each modulus $\fm$ we can attach
the corresponding Takagi group
$$ T_{K/F}\{\fm\} = \{\fa \in D_\fm: \fa = N_{K/F} \fA\} \cdot P_\fm^{(1)}, $$
where $\fA$ runs through the ideals coprime to $\fm$ in $K$, 
and $P_\fm^{(1)}$ the group of principal ideals $(\alpha)$ with
$\alpha \equiv 1 \bmod \fm$. 

The first steps in any approach to class field theory are the
two basic inequalities:
\begin{enumerate}
\item First Inequality\footnote{In Chevalley's idelic approach
      the order of these inequalities is reversed.}.
      For any finite extension $K/F$ and any modulus $\fm$ we have
      $$ (D_\fm : T_{K/F}\{\fm\}) \le (K:F). $$ 
\item Second Inequality. If $K/F$ is a cyclic extension, then
      $$ (D_\fm : T_{K/F}\{\fm\}) \ge (K:F) $$ 
      for any sufficiently large modulus $\fm$.
\end{enumerate}
Once class field theory is established it turns out that the
first inequality can be improved to $(D_\fm : T_{K/F}\{\fm\}) \mid (K:F)$,
and that the second inequality holds for arbitrary abelian extensions.
Moreover, the smallest modulus $\fm$ for which the second inequality
holds is called the conductor of the extension $K/F$, and the second
inequality holds if and only if the modulus $\fm$ is a multiple of
the conductor.

\subsection*{Cyclotomic Fields}

Cyclotomic fields $K = \Q(\zeta_m)$ are abelian extensions of $\Q$, 
hence class fields. In fact, the Takagi group of $K/\Q$ for the 
modulus $m \infty$ is the group $T_K$ of ideals 
$(N_{K/\Q} \fa) \cdot \Q^1_{m\infty}$. Since norms of ideals
coprime to $m$ are congruent to $1 \bmod m$, we have 
$T_K = \Q^1_{m\infty}$.

\subsection*{Quadratic Fields}
For quadratic number fields $K = \Q(\sqrt{d}\,)$ with discriminant $d$
we take $\fm = (d)$ if $d > 0$ and $\fm = (d)\infty$ if $d < 0$. Then 
$T_{K/\Q} \{\fm\}$ consists of all fractional ideals in $\Z$ that can
be written as a product of a norm of an ideal from $K$ and a 
principal ideal $(a)$ with $a \equiv 1 \bmod \fm$.

Now we claim

\begin{prop}\label{PTakQ}
Let $K$ be a quadratic number fields with discriminant $d$, and set
$$ \fm = \begin{cases}  (d)      & \text{ if } d > 0, \\
                      (d) \infty & \text{ if } d < 0
         \end{cases} $$
Then 
$$ T_{K/\Q} \{\fm\} = \{(a): (\textstyle \frac{d}a) = +1 \}. $$
\end{prop}

\begin{proof}
By definition we have $T_{K/\Q} \{\fm\} = \{N_{K/\Q} \fa\} \cdot P_\fm^{(1)}$,
where $\fa$ runs through the ideals coprime to $\fm$ in $K$. Since a
prime number $p \nmid d$ is the norm of an ideal from $K$ if and only 
if $(\frac dp) = +1$, the prime factors of $N\fa$ all satisfy
$(\frac dp) = +1$. This shows that
$$  T_{K/\Q} \{\fm\}  \subseteq H_\fm = 
           \{(a): (\textstyle \frac{d}a) = +1 \} \subseteq D_\fm. $$
Since $(D_\fm:H_\fm) = 2$ we find that $(D_\fm:T_{K/\Q}\{\fm\}) \ge 2$,
and the first inequality now implies $D_\fm  = T_{K/\Q}\{\fm\}$.
\end{proof}

For proving Prop. \ref{PTakQ} directly one would have to show that 
every positive integer $a$ coprime to $d$ and satisfying
$(\frac da) = +1$ can be written as a product $a = rs$, where
$r$ (if $d < 0$ then we demand $r > 0$) is a product of primes $p$ 
with $(\frac dp) = +1$, and $s \equiv 1 \bmod d$.

\subsection*{Artin's Reciprocity Law}

Now let $H$ be an ideal group defined mod $\fm$; by class field 
theory there is a class field $K/F$ with Galois group 
$\Gal(K/F) \simeq D_\fm/H$. The Artin reciprocity map gives an 
explicit and canonical isomorphism: for a prime ideal $\fp$ in 
$F$ unramified in $K$ let $(\frac{K/F}{\fp})$ denote the
Frobenius automorphism of a prime ideal $\fP$ in $K$ above $\fp$.

\begin{thm}[Artin's Reciprocity Law]
The symbol $\big(\frac{K/F}{\fa}\big)$ only depends on the ideal class
of $\fa$ in $D_\fm/I$, and the Artin map induces an exact sequence
$$\begin{CD} 1 @>>> H_m @>>>  D_m @>>> \Gal(K/F) @>>> 1. \end{CD} $$
In particular, there is an isomorphism
$$ D_\fm/H \lra \Gal(K/F) $$
between the generalized ideal class group $D_\fm/H$ and the
Galois group $\Gal(K/F)$ as abelian groups.
\end{thm}

We remark in passing that if $F/k$ is a normal extension, then
the Artin isomorphism is an isomorphism of $\Gal(k/F)$-modules.

\subsection*{Quadratic Reciprocity I}
Already the fact that the Artin map in quadratic number fields 
is constant on the cosets of $D_\fm/H$ implies the quadratic 
reciprocity law. In fact, we may identify the Artin symbol
$(\frac{K/\Q}{p})$ with the Kronecker symbol $(\frac dp)$, 
and the claim that this only depends on the class generated
by $(p)$ in $D_\fm/T_{K/\Q}\{(d)\infty\}$ (compare Prop. \ref{PTakQ})
means that $(\frac dp) = (\frac dq)$ for all primes 
$p \equiv q \bmod d\infty$. This is Euler's formulation of the
quadratic reciprocity law, which in turn is easily seen to be
equivalent to Legendre's version (see \cite{LBT}).

\subsection*{Quadratic Reciprocity II}
Let $p$ be a positive prime and put $p^* = (-1)^{(p-1)/2} p$; then
$p^* \equiv 1 \bmod 4$.
The quadratic number field $K = \Q(\sqrt{p^*}\,)$ is a class field 
modulo $p\infty$ of $\Q$; in fact its associated ideal group is the 
group $I^2P_{p\infty}$, where $I$ is the group of all fractional 
ideals in $\Q$ coprime to $p$, and 
$P_{p\infty} = \{(a) \in I: a \equiv 1 \bmod p\infty\}$.

Now observe that an odd prime $q \ne p$ splits 
\begin{itemize}
\item in the Kummer extension
      $K/\Q$ if and only if $(p^*/q) = +1$;
\item in the class field $K/\Q$ if and only if
      $q \equiv a^2 \bmod p\infty$, that is,
      if and only if $q > 0$ and $(q/p) = +1$.
\end{itemize}
Comparing these decomposition laws implies the quadratic reciprocity 
law $(p^*/q) = (q/p)$ for positive primes $p, q$. Observe that the
Legendre symbol $(q/p)$ giving the decomposition in the class field
can be identified with the Artin symbol for the quadratic extension,
which in turn is defined as the Frobenius automorphism of $K/\Q$.

\section{The Transfer Map}

In this section we will show that the transfer map, which was
first defined by Schur, shows up naturally in class field theory.
In the next section we will show that Gauss's Lemma in the theory
of quadratic residues is an example of a transfer, and give a
class field theoretical interpretation of its content. As an 
application, we will prove the quadratic reciprocity law using
the transfer in the incarnation of Gauss's Lemma.

Consider the following situation: let $F$ be a number field,
$L/F$ an abelian extension, and $K/F$ some subextension. Let 
$G = \Gal(L/F)$ and set $U = \Gal(L/K)$. Then $L$ is a class 
field of $K$ with respect to some ideal group $T_{L/K}$, and 
this ideal group is contained in $T_{L/F}$, which describes 
$L$ as a class field over $K$. Ideals in the principal class 
in $F$ are also contained in the principal class in $K$.

It is therefore a natural question to ask how the Artin symbols
$\art{L/F}{\fa}$ and $\art{L/K}{\fa}$ are related\footnote{Observe
that the ideal $\fa$ in $\art{L/F}{\fa}$ should more exactly be written
as $\fa \cO_F$. In particular, one has to recall that in the symbol
$\art{L/F}{\fp}$, the prime ideal $\fp$ in $K$ need not be prime anymore 
in $F$}. We know that the Artin map induces isomorphisms
$$ \Big(\frac{L/F}{\cdot}\Big): \Cl(F) \lra G = \Gal(L/F)
   \quad \text{and} \quad 
   \Big(\frac{L/K}{\cdot}\Big): \Cl(K) \lra U = \Gal(L/K). $$   

\begin{figure}[h!]
\begin{diagram}%[height=0.8cm,width=0.6cm]
        &              & L       & &             & & & & 1 \\ 
        & \ruLine      & \uLine  & &             & & & \ruLine & \uLine \\
  F'    &              & K       & &
                   \rTo^{\Gal(L/\ \cdot\ )} & & H & & U  \\
\uLine  &  \ruLine     &         & &             & & \uLine & \ruLine  &  \\ 
  F     &              &         & &             & & G  &  &        
\end{diagram}
\caption{}\label{Fig1}
\end{figure}

It follows from the above that the map sending 
$\art{L/F}{\fa} \too \art{L/K}{\fa}$ induces a group homomorphism
$V_{G \to U}: G \lra  U$. The homomorphism $V_{G \to U}$ is called the
transfer  (Verlagerung), and can be computed explicitly as follows:
let $G = \bigcup r_i U$ be a decomposition of $G$ into disjoint
cosets $r_j U$. For $g \in G$ write
$ gr_i = r_j u_j$ for $u_j \in U_j $, and set 
$$ V_{G \to U} (g) = \prod_j u_j \cdot U', $$
where $U'$ is the commutator subgroup of $U$.

\subsection*{Transfer in Abelian Groups}

Although the transfer map is usually connected with nonabelian 
groups since the transfer $G/G' \lra G'/G''$, where 
$G = Gal(K^\infty/K)$ and $K^\infty$ is the Hilbert class field of $K$, 
is used for studying the capitulation of ideal classes from $K$ in 
subfields of its Hilbert class field $K^1$, the transfer can also be 
used for abelian groups:

\begin{lem}\label{GTc}
Let $U$ be a subgroup of an abelian group $G$, and assume that 
$G/U$ is a cyclic group of order $f$. Then $V_{G \to U}(x) = x^f$
for all $x \in G$.
\end{lem}

\begin{proof}
Let $zU$ be a generator of $G/U$, and write 
$G = U \cup zU \cup \ldots \cup z^{f-1}U$. Given any
$x \in G$, we can write $x = z^ju$ for some $u \in U$ and find
$V(x) = V(z^ju) = V(z)^j V(u)$; thus we only need to compute
$V(u)$ and $V(z)$. 
\begin{enumerate}
\item Since $u \cdot x^j = x^j h_j(u)$ for $h_j(u) = u$,
      we find $V(u) = \prod h_j(u) = u^f$. 
\item Here $z \cdot z^j = z^{j+1}$, hence $h_j(z) = 1$ for 
      $1 \le j \le f-2$ and $h_{f-1}(z) = z^f \in U$. Thus $V(z) = z^f$.
\end{enumerate}
Now $V(x) = V(z^ju) = z^{jf}u^f = x^f$ as claimed.
\end{proof}

If $U$ is a subgroup of Klein's four group $G$ with order $2$, then 
$V_{G \to U}$ is the trivial map. There are also examples where
$V$ is surjective:

\begin{cor}\label{Csur}
If $U$ is a subgroup of a finite cyclic group $G$, then the transfer
$V_{G \to U}$ is surjective.
\end{cor}

\begin{proof}
Let $G = \la g \ra$; then $V(g) = g^{(G:U)}$ generates a subgroup
of order $\# U$, and since $U$ is the unique subgroup of $G$ with 
this order, the claim follows.
\end{proof}

Now assume that $L/F$ is a cyclic extension, and let $K/F$ be the 
subextension fixed by some subgroup $U$ of $G = \Gal(L/F)$ (see 
Fig.~\ref{Fig1}). The transfer map $V_{G \to U}$ is surjective by 
Cor.~\ref{Csur}, and its kernel $H$ fixes a subextension $F'/F$. 
Observe that $\Gal(F'/F) \simeq G/H \simeq U \simeq \Gal(L/K)$. 
Thus the transfer map gives us an isomorphism between these 
two Galois groups:

\begin{prop}\label{Pker}
Let $L/F$ be a cyclic extension as above, and let $\fp$ be a 
prime ideal in $F$ unramified in $L$. Then $V_{G \to U}$ induces
an isomorphism 
$$ G/H \simeq \Gal(F'/F) \lra \Gal(L/K) \simeq U. $$ 
In particular, the order of $\sigma_\fp \cdot H$ is equal to the
order of $V(\sigma_\fp)$; this implies in particular that 
$\fp \in \Spl(F'/F)$ if and only if $\sigma_\fp \in \ker V$. 
\end{prop}

\section{Gauss's Lemma and The Quadratic Reciprocity Law}

Gauss gave eight proofs of the quadratic reciprocity law: his first
proof was by induction and used the existence of auxiliary primes
discussed in \cite{LHa1}. Gauss's second 
proof was based on the genus theory of binary quadratic forms and
became the role model for Kummer's proof of the $p$-th power
reciprocity law in $\Q(\zeta_p)$ for regular primes; the
inequality between the number of genera and the number of ambiguous
ideal classes, which was the central point in Gauss's second proof,
became an important tool in Takagi's version of class field theory.

Gauss's third proof used what became known as Gauss's Lemma in the theory
of quadratic residues. In this section, we will show that it has a 
natural class field theoretic interpretation.

Let $p = 2m+1$ be an odd prime and $a$ an integer coprime to $p$.
A half system modulo $p$ is a set $A = \{a_1, \ldots, a_m\}$ of
integers $a_j$ with the property that every coprime residue class
modulo $p$ has a unique representative in $A$ or in $-A$. Thus
for every $1 \le j \le m$ we can write 
$a \cdot a_j \equiv s_j a_{\pi(j)} \bmod p$, where $s_j \in \{\pm 1\}$
and where $\pi$ is a permutation of $A$. Gauss's Lemma says that
the Legendre symbol $(\frac ap)$ is just the product $\prod s_j$ of
signs in these $m$ congruences.

Now consider the ideal group $H$ consisting of all ideals $(a)$ with 
$a > 0$ and $a \equiv 1 \bmod p$. The corresponding class field
$\Q\{\fm\}$ with $\fm = p \infty$ is $K = \Q(\zeta)$, the field of 
$p$-th roots of unity, and the Artin map sends a residue class 
$a \bmod p\infty$ to the automorphism $\sigma_a: \zeta \too \zeta^a$. 
We can identify the ideal class group $G = D_\fm/H$ with the group 
$(\Z/p\Z)^\times$ by choosing positive generators of ideals.

\begin{lem}\label{LGL}
The subgroup $U$ of $G = (\Z/p\Z)^\times$ generated by the residue 
class $-1 \bmod p$ fixes the maximal real subfield 
$K^+ = \Q(\zeta_p + \zeta_p^{-1})$ of $K = \Q(\zeta_p)$, and the
transfer $V_{G \to U}$ satisfies $V(a) \equiv (\frac ap) \bmod p$.
\end{lem}

\begin{proof}
Since $G$ is abelian, the transfer map $V_{G \to U}$ is a homomorphism 
$V: G \to U$. By Lemma \ref{GTc} we have $V(a) \equiv a^{(p-1)/2} \bmod p$.
The claim now follows from Euler's criterion. 
\end{proof}

The formula for $V(a)$ in Lemma \ref{LGL} is Gauss's Lemma. 
It was observed by various mathematicians that Gauss's Lemma 
is an example of the group theoreticl transfer map; see e.g. 
Cartier \cite{Cart}, Delsarte \cite{Dels}, Leutbecher \cite{Leut} 
and Waterhouse \cite{Wat}. The full class field theoretic 
interpretation given above seems to be new.

Let us now exploit the actual content of the transfer map, namely
the connection with the decomposition of prime ideals given in 
Prop. \ref{Pker}: since $G$ is cyclic, the map $V: G \lra U$ is onto 
by Cor~\ref{Csur}, and $H = \ker V$ is an ideal group whose class 
field is the quadratic subfield $k = \Q(\sqrt{p^*}\,)$ of $K$. By 
Prop. \ref{Pker}, we have $q \in \Spl(k/\Q)$ if and only if 
$V(\sigma_q) = (\frac qp) = 1$; since $q$ splits in $k/\Q$ if and 
only if $(\frac{p^*}q) = 1$, we deduce 

\begin{thm} The Quadratic Reciprocity Law: for two distinct
odd primes $p$ and $q$, we have 
$$  \bigg(\frac{p^*}q\bigg) = \bigg(\frac qp\bigg). $$
\end{thm}

Observe that this proof does not require the full power of class
field theory; it only uses Artin's reciprocity law for cyclotomic
extensions of $\Q$ (which, as we have seen, was proved by Dedekind), 
standard properties of the Frobenius automorphism, and a calculation 
of the transfer map (Gauss's Lemma). Actually the transfer map can 
be eliminated from our arguments, giving the well known proof of the 
reciprocity law by comparing the splitting of primes in quadratic and 
cyclotomic extensions. The proof above showing that Gauss's lemma 
comes up more or less naturally in the class field theoretical context 
has a certain charm, however, and perhaps helps to explain {\em why}
Gauss's Lemma can be used for proving the quadratic reciprocity law.

In fact, most of the elementary proofs of the quadratic reciprocity
law used some form of Gauss's Lemma (see \cite{LBT}), and backed with 
a lot of slicker proofs (see e.g. \cite{LRL}) it is quite easy to look 
down upon them; maybe the exposition above can help to rehabilitate
Gauss's Lemma somewhat.

\end{document}